\documentclass{article}
\usepackage{amsfonts}
\usepackage{amsthm,amsmath,amssymb,amscd,verbatim}
\usepackage{lineno}
\usepackage{graphicx}
\usepackage{epsfig}
\usepackage{color}

\newtheorem{theorem}{Theorem}
\newtheorem{definition}{Definition}

\newtheorem{corollary}[definition]{Corollary}

\newtheorem{example}[definition]{Example}

\newtheorem{proposition}[definition]{Proposition}
\newtheorem{remark}[definition]{Remark}

\newcommand{\mclib}{\color{red}}
\newcommand{\mclie}{\color{black}}

\input{tcilatex}
\begin{document}

\title{Covering relations for coupled map networks}
\author{Leonid Bunimovich\thanks{%
School of Mathematics, Georgia Institute of Technology, 686 Cherry Street,
Atlanta, GA 30332-0160, USA, E-mail: bunimovh@math.gatech.edu} , Ming-Chia Li%
\thanks{%
Department of Applied Mathematics, National Chiao Tung University, 1001 Ta
Hsueh Road, Hsinchu 300, TAIWAN, Tel:\ +866-3-5712121 ext. 56463, Fax:\
+866-3-5131223, E-mail: mcli@math.nctu.edu.tw} \ and Ming-Jiea Lyu\thanks{%
Department of Applied Mathematics, National Chiao Tung University, 1001 Ta
Hsueh Road, Hsinchu 300, TAIWAN}}
\maketitle

\begin{abstract}
Following \cite{KL10, Y08}, we study coupled map networks over arbitrary
finite graphs. An estimate from below for a topological entropy of a
perturbed coupled map network via a topological entropy of an unperturbed
network by making use of the covering relations for coupled map networks is
obtained. The result is quite general, particularly no assumptions on
hyperbolicity of a local dynamics or linearity of coupling are made.\medskip 

\medskip 

\end{abstract}

\section{Introduction}

A coupled map network is characterized by local dynamics operating at each
node of a graph and their interaction along the edges of the graph. Coupled
map networks are useful in applications:\ they appear naturally as
electronic circuits in engineering, as chemical reactions in physics, as
neuronal networks in the biological sciences, and as various agent-based
models in the social sciences, etc. These networks are usually finite in
size, but can be very large. The couplings of them can be linear as well as
nonlinear. Much attention has been paid to linear coupling and simple
dynamical behaviors such as existence of a global attractor in \cite{AB07,
ABM10} and synchronization in \cite{BAJ09, F08}; see also \cite{CF05} and
references therein, while more complex phenomena are known to occur but not
as well recognized and understood. Recently in \cite{KL10}, a linear
coupling of expanding circle maps was studied. It was found there that an
increasing of coupling strength leads to a cascade of bifurcations in which
unstable subspaces in the coupled map systematically become stable.

We study here a topological dynamics in coupled map networks without
assuming hyperbolicity of local maps and linearity of interactions. Consider
a coupled map network with local dynamics having covering relations and
coupling having linear model. We give sufficient conditions for existence of
periodic points and for existence of a positive topological entropy, in the
coupled map network. Both conditions allow for a weak as well as for a
strong coupling. Moreover, both results are also valid for small
perturbations of the coupled map network, of which coupling might be
nonlinear.

Our approach is based on the concept of covering relations, introduced by 
\cite{ZG04}. The covering relation is a topological technique which does not
require hyperbolicity (see e.g. \cite{MZ01, LL11, LLZ08}). Assuming that a
coupling is locally topologically conjugate to a linear coupling, we show
that the unperturbed coupled map network and its small perturbations both
have covering relations of local dynamics as well as existence of periodic
points. To implement a topological chaos from local dynamics to perturbed
coupled map networks, we introduce a notion of unified sets to guarantee the
conjugacy relation between the coupling and its linear model.

The paper is organized as follows. In Section 2, we define coupled map
networks (here and in other places) of two types and state main results for
each of them. In Section 3, we begin with formulation of known results
needed in what follows and present the proofs of our results. In Appendix,
the definition of covering relations determined by a transition matrix is
briefly recalled.

\section{Definitions and statements of the main results}

We start with the definition of a general class of coupled map networks
which will be studied.

\begin{definition}
A coupled map network is a triple $(G,\{T_{k}\},A)$ where

\begin{enumerate}
\item $G$ is a connected directed graph specified by a finite set $\Omega $
of nodes and a collection of edges $\mathcal{E}\subset \Omega \times \Omega $%
;

\item to each node $k\in \Omega $ there corresponds a local space $X_{k}$
and a local map $T_{k}:X_{k}\rightarrow X_{k};$

\item network dynamics is defined by the iteration of $\Phi :X\rightarrow X$%
, where $X=\prod_{k\in \Omega }X_{k}$ is the product space and $\Phi =A\circ
T$, where $T=\prod_{k\in \Omega }T_{k}$ is the (independent) application of
local maps and $A:X\rightarrow X$ is the spatial interaction or coupling;
for $x=(x_{k})_{k\in \Omega }\in X$, the $k$th coordinate of $A(x)$ depends
only on $x_{k}$ and those $x_{j}$ for which $(j,k)\in \mathcal{E}$.
\end{enumerate}
\end{definition}

We consider perturbations of coupled map networks in the following sense.

\begin{definition}
Let $(G,\{T_{k}\},A)$ and $(\tilde{G},\{ \tilde{T}_{k}\},\tilde{A})$ be two
coupled map networks with the same set of nodes (might with distinct edges)
and local spaces. For $\varepsilon >0$, we say that $(\tilde{G},\{ \tilde{T}%
_{k}\},\tilde{A})$ is $\varepsilon $-close to $(G,\{T_{k}\},A)$ if $|\tilde{A%
}(z)-A(z)|<\varepsilon $ for all $z\in X$ and $|\tilde{T}_{k}(x)-T_{k}(x)|<%
\varepsilon $ for all $x\in X_{k}$ for all nodes $k$, where $|\cdot |$
denotes norms on the product space and on the local spaces without ambiguity.
\end{definition}

Given a coupled map network $(G,\{T_{k}\},A)$, if $(\tilde{G},\{ \tilde{T}%
_{k}^{\varepsilon }\},\tilde{A}^{\varepsilon })$ is a family of coupled map
networks of the same nodes and local spaces, with a real parameter $%
\varepsilon $, such that $\tilde{T}_{k}^{\varepsilon
}(x_{k})=T_{k}(x_{k})+\varepsilon \alpha (x_{k})$ and $\tilde{A}%
^{\varepsilon }(x)=A(x)+\varepsilon \beta (x)$, where both $\alpha $ and $%
\beta \ $are bounded and continuous functions, then $(\tilde{G},\{ \tilde{T}%
_{k}^{\varepsilon }\},\tilde{A}^{\varepsilon })$ is approaching $%
(G,\{T_{k}\},A)$ as $\varepsilon $ tends to zero.

For a positive integer $m$, let $\mathbb{R}^{m}$ denote the space of all $m$%
-tuples of real numbers. Let $\left \vert \cdot \right \vert $ be a given norm
on $\mathbb{R}^{m}$, and let $\left \Vert \cdot \right \Vert $ denote the
operator-norm on the space of linear maps on $\mathbb{R}^{m}$ induced by $%
\left \vert \cdot \right \vert $. For $x\in \mathbb{R}^{m}$ and $r>0,$ we
denote $B^{m}(x,r)=\{z\in \mathbb{R}^{m}:\left \vert z-x\right \vert <r\}$;
for the particular case when $x=0$ and $\ r=1$, we write $B^{m}=B^{m}(0,1)$,
that is, the open unit ball in $\mathbb{R}^{m}$. Furthermore, for any subset 
$S$ of $\mathbb{R}^{m},$ let $\overline{S},$ $\mathtt{int}(S)$ and $\partial
S$ denote the closure, interior and boundary of $S$, respectively. For the
definition of \emph{covering relations determined by a transition matrix},
see Appendix.

\begin{definition}
Let $F$ be a continuous map on $\mathbb{R}^{m}$. Define the maximum stretch $%
||F||_{\max }=\max \{|F(x)|:\ x\in \overline{B^{m}}\}$ and the minimum
stretch $||F||_{\min }=\min \{|F(x)|:\ x\in \partial B^{m}\}$.
\end{definition}

The maximum and minimum stretches are the radii of the smallest ball with
center at origin that contains $F(\overline{B^{m}})$ and of the largest open
ball with center at origin not intersecting $F(\partial B^{m})$. If $F$ is a
linear map, the maximum and minimum stretches are the norm and conorm of $F.$

From now on, we consider a coupled map network $(G,\{T_{k}\},A)$ such that $%
G $ is a connected directed graph with nodes $\Omega =\{1,\ldots ,d\}$ and
edges $\mathcal{E}\subset \Omega \times \Omega $, and for $1\leq k\leq d$, $%
T_{k}$ is a continuous local map on $\mathbb{R}^{u+s}$ having covering
relations on h-sets $\{M_{ki}\}_{i=1}^{d_{k}}$ determined by a transition
matrix $W_{k}=[w_{kij}]_{1\leq i,j\leq d_{k}}$ such that $u(M_{ki})=u$ and $%
s(M_{ki})=s$.

We say that the coupled map network $(G,\{T_{k}\},A)$ is of \emph{type I}
with \emph{locally linear coupling}, if for each nonzero entry $%
\prod_{k=1}^{d}w_{ki_{k}j_{k}}$ of the Kronecker product $%
\bigotimes_{k=1}^{d}W_{k}$, there exists a $d\times d$ invertible real
matrix $[a_{lm}]$ satisfying $(m,l)\notin \mathcal{E}$ implies $a_{lm}=0$,
and the following conditions hold:\ 

\begin{itemize}
\item For $1\leq k\leq d$ and $1\leq i^{\prime },j^{\prime }\leq d_{k}$ with 
$i^{\prime }\neq i_{k}$ and $j^{\prime }\neq j_{k},$%
\begin{equation}
T_{k}(M_{ki_{k}})\cap (M_{kj^{\prime }}\cup T(M_{ki^{^{\prime
}}}))=\emptyset .  \label{eq: not overlap}
\end{equation}

\item For $1\leq k\leq d$ and $x\in \overline{B^{u}},y\in \overline{B^{s}},$ 
\begin{equation}
c_{M_{kj_{k}}}\circ T_{k}\circ
c_{M_{ki_{k}}}^{-1}(x,y)=(U_{ki_{k}j_{k}}(x),V_{ki_{k}j_{k}}(y)),
\label{eq: UV1}
\end{equation}%
where $U_{ki_{k}j_{k}}$ and $V_{ki_{k}j_{k}}$ are continuous maps on $%
\mathbb{R}^{u}$ and $\mathbb{R}^{s}$, respectively, such that 
\begin{equation}
||U_{ki_{k}j_{k}}||_{\min }>1,\deg (U_{ki_{k}j_{k}j},B^{u},0)\neq 0,\text{
and }||V_{ki_{k}j_{k}}||_{\max }<1.  \label{eq: UVnorm1}
\end{equation}

\item For $z\in (\prod_{k=1}^{d}c_{M_{kj_{k}}})\circ
T(\prod_{k=1}^{d}M_{ki_{k}})$, 
\begin{equation}
(\prod_{k=1}^{d}c_{M_{kj_{k}}})\circ A\circ
(\prod_{k=1}^{d}c_{M_{kj_{k}}})^{-1}(z)=([a_{lm}]\bigotimes I)z,
\label{eq: Ac1}
\end{equation}
\end{itemize}

where $I$ is the $(u+s)\times (u+s)$ identity matrix. 

Notice that (\ref{eq: UV1}) and (\ref{eq: UVnorm1}) are only to specify the
covering relation $M_{ki_{k}}\overset{T_{k}}{\Longrightarrow }M_{kj_{k}}$.
From (\ref{eq: not overlap}) it follows that (\ref{eq: Ac1}) is well defined
and it says that the restriction of $A$ to the set $T(%
\prod_{k=1}^{d}M_{ki_{k}})$ is topologically conjugate to the linear map $%
A_{c}$, by the homeomorphism $\prod_{k=1}^{d}c_{M_{kj_{k}}}$. In general,
the map on the left-hand side of (\ref{eq: Ac1}) is not well defined.

Now, we state the first result about covering relations and existence of
periodic points for perturbations of coupled map networks, under permutation
transition matrices.

\begin{theorem}
\label{thm: main-periodic point} Let $(G,\{T_{k}\},A)$ be a coupled map
network of type I with locally linear coupling as in (\ref{eq: Ac1}) such
that each of the transition matrices $W_{k}$, $1\leq k\leq d$, is a
permutation. Suppose that for each nonzero entry $%
\prod_{k=1}^{d}w_{ki_{k}j_{k}}$ of the Kronecker product $%
\bigotimes_{k=1}^{d}W_{k}$, there exists a permutation $\tau $ on $%
\{1,2,\ldots ,d\}$ such that for $1\leq k\leq d$, 
\begin{equation}
||a_{k\tau (k)}U_{\tau (k)i_{\tau (k)}j_{\tau (k)}}||_{\min
}-\sum_{l=1,l\neq \tau (k)}^{d}||a_{kl}U_{li_{l}j_{l}}||_{\max }>1\text{ and 
}\sum_{l=1}^{d}||a_{kl}V_{li_{l}j_{l}}||_{\max }<1.  \label{eq: thm1}
\end{equation}%
Then any coupled map network $(\tilde{G},\{ \tilde{T}_{k}\},\tilde{A})$
sufficiently close to $(G,\{T_{k}\},A)$ has covering relations determined by 
$\bigotimes_{k=1}^{d}W_{k}$ and has a periodic point of period $\mathtt{lcm}%
(\dim W_{1},\ldots ,\dim W_{d}))$, where $\mathtt{lcm}$ stands for the least
common multiple.
\end{theorem}

Before further investigating topological chaos for perturbations of coupled
map networks, we introduce a notion of unified h-sets.

\begin{definition}
A $d$-tuple $(M_{1},\cdots ,M_{d})$ of disjoint h-sets in $\mathbb{R}^{m}$
with $u(M_{i})=u$ is said to be \emph{unified} by a subset $N$ of $\mathbb{R}%
^{m}$ if $\cup _{i=1}^{d}M_{i}\subset N$, and there exists a homeomorphism $%
\hat{c}_{N}:\mathbb{R}^{m}\rightarrow \mathbb{R}^{m}=\mathbb{R}^{u}\times 
\mathbb{R}^{m-u}$ such that for $1\leq i\leq d$, 
\begin{equation}
\hat{c}_{N}(N)=\overline{B^{u}}(q^{u},(3d-1)/2)\times \overline{B^{s}}(0,1)%
\text{ and }\hat{c}_{N}(M_{i})=\overline{B^{u}}(q_{i}^{u},1)\times \overline{%
B^{s}}(q_{i}^{s},r_{i}),  \label{eq: unify}
\end{equation}%
where $q^{u}=((3d-3)/2,0,...,0)$ and $q_{i}^{u}=(3(i-1),0,...,0)$ belong to $%
\mathbb{R}^{u}$, and $q_{i}^{s}=(\tilde{q}_{i}^{s},0,...,0)$ belongs to $%
\mathbb{R}^{s}$ for some real numbers $|\tilde{q}_{i}^{s}|<1$ and $%
0<r_{i}\leq 1$. Here, we call $q_{i}^{u}$, $q_{i}^{s}$, and $r_{i}$, the $i$%
th \emph{unstable center}, \emph{stable center}, and \emph{radius of
stability} of $N$ respectively. In particular, any h-set is unified by
itself.
\end{definition}

Being unified is a topological aspect: a unified tuple means the union of
elements in the tuple is an enlarged h-set such that under the change of
coordinates as (\ref{eq: unify}), the union looks like a product of unstable
and stable balls, while the choice of centers and radii is flexible.

Let $A_{c}$ be the following Kronecker product 
\begin{equation}
A_{c}=[a_{lm}]\bigotimes I,  \label{eq: AC-unify}
\end{equation}%
where $[a_{lm}]$ is a $d\times d$ invertible real matrix such that if $%
(l,m)\notin \mathcal{E}$ then $a_{lm}=0$, and $I$ is the $(u+s)\times (u+s)$
identity matrix. We say that the coupled map network $(G,\{T_{k}\},A)$ is of 
\emph{type II }with the \emph{linear coupling model }$A_{c}$, if the
following conditions hold:\ 

\begin{itemize}
\item For $1\leq k\leq d$, there exists a set $N_{k}$ such that the tuple $%
(M_{k1},...,M_{kd_{k}})$ is unified by $N_{k}$ with the $j$th unstable
center at $p_{kj}^{u}$, stable center at $p_{kj}^{s}$, and stable radius $%
r_{kj}$ for all $1\leq j\leq d_{k}$.

\item For $1\leq k\leq d$, if $w_{kij}=1$, then for $x\in \overline{B^{u}}%
,y\in \overline{B^{s}},$ 
\begin{equation}
\hat{c}_{N_{k}}\circ T_{k}\circ \tilde{c}%
_{M_{ki}}^{-1}(x,y)=(U_{ki}(x),V_{ki}(y)),  \label{eq: UV2}
\end{equation}%
where 
\begin{equation}
g_{ki}(x,y)=(x-p_{ki}^{u},(y-p_{ki}^{s})/r_{ki})\text{ and }\tilde{c}%
_{M_{ki}}=g_{ki}\circ \hat{c}_{N_{k}},  \label{eq: qc}
\end{equation}%
and $U_{ki},V_{ki}$ are continuous maps on $\mathbb{R}^{u},\mathbb{R}^{s}$,
respectively, such that 
\begin{equation}
||U_{ki}-p_{kj}^{u}||_{\min }>1,\deg (U_{ki},B^{u},p_{kj}^{u})\neq 0,\text{
and }||V_{ki}-p_{kj}^{s}||_{\max }<r_{kj}.  \label{eq: UVnorm2}
\end{equation}

\item For $z\in (\prod_{k=1}^{d}\hat{c}_{N_{k}})\circ
T(\prod_{k=1}^{d}M_{ki}),$%
\begin{equation}
(\prod_{k=1}^{d}\hat{c}_{N_{k}})\circ A\circ (\prod_{k=1}^{d}\hat{c}%
_{N_{k}})^{-1}(z)=A_{c}z.  \label{eq: Ac2}
\end{equation}
\end{itemize}

Notice that (\ref{eq: UV2}) and (\ref{eq: UVnorm2}) are only to specify the
covering relation $M_{ki}\overset{T_{k}}{\Longrightarrow }M_{kj}$ under the
unified structure. With a help of (\ref{eq: qc}), each quadruple $(M_{ki},%
\tilde{c}_{M_{ki}},u,s)$ now is an h-sets in $\mathbb{R}^{u+s}$. Moreover,
since $(\prod_{k=1}^{d}\hat{c}_{N_{k}})$ is independent of $i$ and $j$, (\ref%
{eq: Ac2}) is always well defined and it says that the restriction of $A$ to
the set $T(\prod_{k=1}^{d}M_{ki})$ is topologically conjugate to the linear
map $A_{c}$, by the homeomorphism $(\prod_{k=1}^{d}\hat{c}_{N_{k}})$.

We give examples of coupled map networks of type II:\ one for $A=A_{c}$ and
the other for $A\neq A_{c}$.

\begin{example}
Define local dynamics by 
\begin{equation*}
T_{1}(x)=\left \{ 
\begin{array}{ll}
3.5x+1.5, & \text{if }x\leq 3/2, \\ 
2x-6, & \text{if }x>3/2,%
\end{array}%
\right. \text{ and }T_{2}(x)=\left \{ 
\begin{array}{ll}
2x+3, & \text{if }x\leq 3/2, \\ 
3.5x-9, & \text{if }x>3/2.%
\end{array}%
\right.
\end{equation*}%
Then $T_{1}$ has covering relations determined by $W_{1}=\left[ 
\begin{array}{cc}
1 & 1 \\ 
1 & 0%
\end{array}%
\right] $ on the h-set tuple $(M_{11}=[-1,1],M_{12}=[2,4]),$ with $%
c_{M_{11}}(x)=x,c_{M_{12}}(x)=x-3$ and $u=1,s=0$, which is unified by $%
N_{1}=[-1,4]$ with $\hat{c}_{N_{1}}(x)=x$ and unstable centers at $%
p_{11}^{u}=0$ and $p_{12}^{u}=3$. Also, $T_{2}$ has covering relations
determined by $W_{2}=\left[ 
\begin{array}{cc}
0 & 1 \\ 
1 & 1%
\end{array}%
\right] $ on the h-set tuple $(M_{21}=[-1,1],M_{22}=[2,4]),$ with $%
c_{M_{21}}(x)=x,c_{M_{22}}(x)=x-3$ and $u=1,s=0$, which is unified by $%
N_{2}=[-1,4]\ $with $\hat{c}_{N_{2}}(x)=x$ and unstable centers at $%
p_{11}^{u}=0$ and $p_{12}^{u}=3$. For $-1\leq x\leq 1$, let $%
U_{11}(x)=3.5x+1.5,U_{12}(x)=2x,U_{21}(x)=2x+3,$ and $U_{22}(x)=3.5x+1.5.$
Then (\ref{eq: UV2}) and (\ref{eq: UVnorm2}) hold. Let $G$ be a complete
graph with two nodes and define a coupling by $%
A(x,y)=(a_{11}x+a_{12}y,a_{21}x+a_{22}y)$. Then, $(G,\{T_{k}\},A)$ is of
type II with the linear coupling model $A_{c}=A$.
\end{example}

\begin{example}
Define local dynamics by 
\begin{equation*}
T_{1}(x)=\left \{ 
\begin{array}{ll}
3.5x-1, & \text{if }x\leq 5/2, \\ 
2x-7, & \text{if }x>5/2,%
\end{array}%
\right. \text{ and }T_{2}(x)=\left \{ 
\begin{array}{ll}
2x+1, & \text{if }x\leq 7/2, \\ 
3.5x-14, & \text{if }x>7/2.%
\end{array}%
\right.
\end{equation*}%
Then $T_{1}$ has covering relations determined by $W_{1}=\left[ 
\begin{array}{cc}
1 & 1 \\ 
1 & 0%
\end{array}%
\right] $ on the h-set tuple $(M_{11}=[0,2],M_{12}=[3,5]),$ with $%
c_{M_{11}}(x)=x-1,c_{M_{12}}(x)=x-4$ and $u=1,s=0$, which is unified by $%
N_{1}=[0,5]$ with $\hat{c}_{N_{1}}(x)=x-1$ and unstable centers at $%
p_{11}^{u}=0$ and $p_{12}^{u}=3$. Also, $T_{2}$ has covering relations
determined by $W_{2}=\left[ 
\begin{array}{cc}
0 & 1 \\ 
1 & 1%
\end{array}%
\right] $ on the h-set tuple $(M_{21}=[1,3],M_{22}=[4,6]),$ with $%
c_{M_{21}}(x)=x-2,c_{M_{22}}(x)=x-5$ and $u=1,s=0$, which is unified by $%
N_{2}=[1,6]\ $with $\hat{c}_{N_{2}}(x)=x-2$ and unstable centers at $%
p_{11}^{u}=0$ and $p_{12}^{u}=3$. For $-1\leq x\leq 1$, let $%
U_{11}(x)=3.5x+1.5,U_{12}(x)=2x,U_{21}(x)=2x+3,$ and $U_{22}(x)=3.5x+1.5.$
Then (\ref{eq: UV2}) and (\ref{eq: UVnorm2}) hold. Let $G$ be a complete
graph with two nodes and define a coupling by $%
A(x,y)=(a_{11}(x-1)+a_{12}(y-2)+1,a_{21}(x-1)+a_{22}(y-2)+2)$. Then, $%
(G,\{T_{k}\},A)$ is of type II with the linear coupling model $%
A_{c}(x,y)=(a_{11}x+a_{12}y,a_{21}x+a_{22}y)$.
\end{example}

Now, we state our result on covering relations and topological entropy of
perturbed coupled map networks.

\begin{theorem}
\label{thm: main} Let $(G,\{T_{k}\},A)$ be a coupled map network of type II
with the linear coupling model\emph{\ }$A_{c}$ as in (\ref{eq: AC-unify}).
Suppose that for each nonzero entry $\prod_{k=1}^{d}w_{ki_{k}j_{k}}$ of the
Kronecker product $\bigotimes_{k=1}^{d}W_{k}$, there exists a permutation $%
\tau $ on $\{1,\ldots ,d\}$ such that for $1\leq k\leq d$, $%
p_{kj_{k}}^{u}\in a_{k\tau (k)}U_{\tau (k)i_{\tau (k)}}(B^{u})$,and 
\begin{equation}
||a_{k\tau (k)}U_{\tau (k)i_{\tau (k)}}-p_{kj_{k}}^{u}||_{\min
}-\sum_{l=1,l\neq \tau (k)}^{d}||a_{kl}U_{li_{l}}||_{\max }>1,\text{ and }%
\sum_{l=1}^{d}||a_{kl}V_{li_{l}}-p_{kj_{k}}^{s}||_{\max }<r_{kj_{k}}.
\label{eq: thm2}
\end{equation}%
Then any coupled map network $(\tilde{G},\{ \tilde{T}_{k}\},\tilde{A})$
sufficiently close to $(G,\{T_{k}\},A)$ has covering relations determined by 
$\bigotimes_{k=1}^{d}W_{k}$ and has topological entropy bounded below by $%
\mathtt{\log }(\prod_{k=1}^{d}\rho (W_{k}))$.
\end{theorem}

\begin{remark}
Let $\{A^{\varepsilon }\}$ and $\{T^{\varepsilon }\}$ be one-parameter
families of maps on $\mathbb{R}^{(u+s)d}$, where $\varepsilon \in \mathbb{R}$
is a parameter, such that $A^{0}=A,T^{0}=\prod_{k\in \Omega }T_{k}$, and $%
A^{\varepsilon }(z)$ and $T^{\varepsilon }(z)$ are both continuous jointly
in $\varepsilon $ and $z$, then Theorem \ref{thm: main} holds for $%
A^{\varepsilon }\circ T^{\varepsilon }$ if $\varepsilon $ is sufficiently
small.
\end{remark}

\section{Proofs of Theorems \protect \ref{thm: main-periodic point} and 
\protect \ref{thm: main}}

First, we list some known results \cite{LL11, ZG04} which will be needed in
the proofs. The following one ensures persistence of covering relations for $%
C^{0}$ perturbations.

\begin{proposition}
\cite[Proposition 14]{LL11}\label{prop: persistence} Let $M$ and $N$ be
h-sets in $\mathbb{R}^{m}$ with $u(M)=u(N)=u$ and $s(M)=s(N)=s$ and let $%
f,g:M\rightarrow \mathbb{R}^{m}$ be continuous. Assume that $M\overset{f}{%
\Longrightarrow }N.$ Then there exists $\delta >0$ such that if $%
|f(x)-g(x)|<\delta $ for all $x\in M$ then $M\overset{g}{\Longrightarrow }N.$
\end{proposition}

The next statement says that a closed loop of covering relations implies
existence of a periodic point.

\begin{proposition}
\cite[Theorem 9]{ZG04}\label{prop: closed loop} Let $\{f_{i}\}_{i=1}^{k}$ be
a collection of continuous maps on $\mathbb{R}^{m}$ and $\{M_{i}\}_{i=1}^{k}$
be a collection of h-sets in $\mathbb{R}^{m}$ such that $M_{k+1}=M_{1}$ and $%
M_{i}\overset{f_{i}}{\Longrightarrow }M_{i+1}$ for $1\leq i\leq k.$ Then
there exists a point $x\in \mathtt{int}(M_{1})$ such that%
\begin{align*}
f_{i}\circ f_{i-1}\circ \cdots \circ f_{1}(x)& \in \mathtt{int}(M_{i+1})%
\text{ for }i=1,...k,\text{ and} \\
f_{k}\circ f_{k-1}\circ \cdots \circ f_{1}(x)& =x.
\end{align*}
\end{proposition}

It is known that a\ continuous map having covering relations determined by a
transition matrix is topologically semi-conjugate to a one-sided subshift of
finite type.

\begin{proposition}
\cite[Proposition 15]{LL11}\label{prop: semi-conjugacy} Let $f:\mathbb{R}%
^{m}\rightarrow \mathbb{R}^{m}$ be a continuous map which has covering
relations determined by a transition matrix $W$. Then there exists a compact
subset $\Lambda $\ of $\mathbb{R}^{m}$\ such that $\Lambda $\ is maximal
positively invariant for $f$\ in the union of the h-sets $($with respect to $%
W)$ and $f|\Lambda $\ is topologically semi-conjugate to $\sigma _{W}^{+}$.
\end{proposition}

Finally, we summarize basic properties of the local Brouwer degree; refer to 
\cite[Chapter III]{Sch69} for the proof.

\begin{proposition}
\label{prop: degree} Let $S$ be an open and bounded subset of $\mathbb{R}%
^{m} $ with $m\geq 1$, and let $\varphi :\bar{S}\rightarrow \mathbb{R}^{m}$
be continuous and $q\in \mathbb{R}^{m}$ such that $q\notin \varphi (\partial
S)$. Then the following holds:\ 

\begin{enumerate}
\item \label{item-det-degree}If $\varphi $ is $C^{1}$ and for each $x\in
\varphi ^{-1}(q)\cap S$ the Jacobian matrix of $\varphi $ at $x$, denoted by 
$D\varphi _{x}$, is nonsingular, then 
\begin{equation*}
\deg (\varphi ,S,q)=\sum_{x\in \varphi ^{-1}(q)\cap S}\mathtt{sgn}(\det
D\varphi _{x}),
\end{equation*}%
where $\mathtt{sgn}$ is the sign function.

\item \label{item-multiplication-degree}Let $\psi :\mathbb{R}^{m}\rightarrow 
\mathbb{R}^{m}$ be a $C^{1}$ map and $p\in \mathbb{R}^{m}$ such that $\psi
^{-1}(p)$ consists of a single point and lies in a bounded connected
component $\Delta $ of $\mathbb{R}^{m}\setminus \varphi (\partial S)$, and $%
D\psi _{\psi ^{-1}(p)}$ is nonsingular. Then 
\begin{equation*}
\deg (\psi \circ \varphi ,S,p)=\mathtt{sgn}(\det D\psi _{\psi ^{-1}(p)})\deg
(\varphi ,S,v),
\end{equation*}%
for any $v\in \Delta $.

\item \label{item-product-degree}Let $S^{\prime }$ be an open and bounded
subset of $\mathbb{R}^{n}$ with $n\geq 1$, and let $\psi :\bar{S}^{\prime
}\rightarrow \mathbb{R}^{n}$ be continuous and $q^{\prime }\in \mathbb{R}%
^{n} $ such that $q^{\prime }\notin \psi (\partial S^{\prime })$. Define a
map $(\varphi ,\psi ):\mathbb{R}^{m}\times \mathbb{R}^{n}\rightarrow \mathbb{%
R}^{m}\times \mathbb{R}^{n}$ by $(\varphi ,\psi )(x,y)=(\varphi (x),\psi
(y)) $ for $x\in \mathbb{R}^{m}$ and $y\in \mathbb{R}^{n}$. Then 
\begin{equation*}
\deg ((\varphi ,\psi ),S\times S^{\prime },(q,q^{\prime }))=\deg (\varphi
,S,q)\deg (\psi ,S^{\prime },q^{\prime }).
\end{equation*}
\end{enumerate}
\end{proposition}

Now, we are in a position to prove our main results.

\begin{proof}[Proof of Theorem \protect \ref{thm: main-periodic point}]
Let $\prod_{k=1}^{d}w_{ki_{k}j_{k}}$ be a nonzero entry of $%
\bigotimes_{k=1}^{d}W_{k}$. We shall prove that the following covering
relation holds 
\begin{equation*}
\prod_{k=1}^{d}M_{ki_{k}}\overset{A\circ T}{\Longrightarrow }%
\prod_{k=1}^{d}M_{kj_{k}}.
\end{equation*}%
In the sequel, we use the following notations:$\ x_{k}\in \mathbb{R}^{u}$
and $y_{k}\in \mathbb{R}^{s}$ for $1\leq k\leq d$, $x=\prod_{k=1}^{d}x_{k}%
\in (\mathbb{R}^{u})^{d}=\mathbb{R}^{ud}$, $y=\prod_{k=1}^{d}y_{k}\in (%
\mathbb{R}^{s})^{d}=\mathbb{R}^{sd}$, $\prod_{k=1}^{d}(x_{k},y_{k})\in (%
\mathbb{R}^{u+s})^{d}=\mathbb{R}^{(u+s)d}$, and $(x,y)\in \mathbb{R}%
^{ud}\times \mathbb{R}^{sd}=\mathbb{R}^{(u+s)d}$.

First, we check conditions on h-sets. Let $M=\prod_{k=1}^{d}M_{ki_{k}}$ and $%
N=\prod_{k=1}^{d}M_{kj_{k}}$.\ Then $M$ and\ $N$ are h-sets, with constants $%
u(M)=u(N)=ud$ and $s(M)=s(N)=sd,$ and homeomorphisms $c_{M},c_{N}:\mathbb{R}%
^{(u+s)d}\rightarrow \mathbb{R}^{ud}\times \mathbb{R}^{sd}$, defined as
follows 
\begin{equation*}
c_{M}(\prod_{k=1}^{d}(x_{k},y_{k}))=\varsigma \circ
\prod_{k=1}^{d}c_{M_{ki_{k}}}(x_{k},y_{k}),\text{ and }c_{N}(%
\prod_{k=1}^{d}(x_{k},y_{k}))=\varsigma \circ
\prod_{k=1}^{d}c_{M_{kj_{k}}}(x_{k},y_{k}),
\end{equation*}%
where $\varsigma :\mathbb{R}^{(u+s)d}\rightarrow \mathbb{R}^{(u+s)d}$ is
defined by $\varsigma (\prod_{k=1}^{d}(x_{k},y_{k}))=(x,y)$.

Second, we construct a homotopy such that (\ref{eq:h5})-(\ref{eq:h7}) holds.
Define a homotopy $H:[0,1]\times \overline{B^{ud}}\times \overline{B^{sd}}%
\rightarrow \mathbb{R}^{(u+s)d}$ as 
\begin{equation*}
H(t,x,y)=(1-t)c_{N}\circ A\circ T\circ c_{M}^{-1}(x,y)+t\pi _{1}\circ
c_{N}\circ A\circ T\circ c_{M}^{-1}(x,y)
\end{equation*}%
where $\pi _{1}:\mathbb{R}^{(u+s)d}\rightarrow \mathbb{R}^{(u+s)d}$ is
defined by $\pi _{1}(x,y)=(x,0)$ for all $x\in \mathbb{R}^{ud}$ and $y\in 
\mathbb{R}^{sd}$. Clearly, (\ref{eq:h5}) holds.

Before checking (\ref{eq:h6}) and (\ref{eq:h7}) (see Appendix), we derive a
new form for the homotopy. Define $\bar{c}_{M}=\prod_{k=1}^{d}c_{M_{ki_{k}}}$
and $\bar{c}_{N}=\prod_{k=1}^{d}c_{M_{kj_{k}}}$. Then $c_{M}=\varsigma \circ 
\bar{c}_{M}$ and $c_{N}=\varsigma \circ \bar{c}_{N}$. Let $(x,y)\in 
\overline{B^{ud}}\times \overline{B^{sd}}=M_{c}$. Then $\bar{c}_{N}\circ
T\circ \bar{c}_{M}^{-1}(\prod_{k=1}^{d}(x_{k},y_{k}))\in \bar{c}_{N}(T(M))$
and $\bar{c}_{N}\circ T\circ \bar{c}_{M}^{-1}(\prod_{k=1}^{d}(x_{k},y_{k}))=%
\prod_{k=1}^{d}(U_{ki_{k}j_{k}}(x_{k}),V_{ki_{k}j_{k}}(y_{k}))$. Moreover,
by the definition of $A_{c}$, we obtain that 
\begin{eqnarray*}
&&c_{N}\circ A\circ T\circ c_{M}^{-1}(x,y) \\
&=&\varsigma \circ \bar{c}_{N}\circ A\circ T\circ \bar{c}_{M}^{-1}\circ
\varsigma ^{-1}(x,y)=\varsigma \circ \bar{c}_{N}\circ A\circ \bar{c}%
_{N}^{-1}\circ \bar{c}_{N}\circ T\circ \bar{c}_{M}^{-1}(%
\prod_{k=1}^{d}(x_{k},y_{k})) \\
&=&\varsigma \circ A_{c}\circ \bar{c}_{N}\circ T\circ \bar{c}%
_{M}^{-1}(\prod_{k=1}^{d}(x_{k},y_{k}))=\varsigma \circ A_{c}\circ \prod
\limits_{k=1}^{d}(U_{ki_{k}j_{k}}(x_{k}),V_{ki_{k}j_{k}}(y_{k})) \\
&=&\varsigma \circ ([a_{lk}]\bigotimes I)\circ \prod
\limits_{k=1}^{d}(U_{ki_{k}j_{k}}(x_{k}),V_{ki_{k}j_{k}}(y_{k}))=\varsigma
\circ \prod \limits_{l=1}^{d}(\sum
\limits_{k=1}^{d}a_{lk}U_{ki_{k}j_{k}}(x_{k}),\sum
\limits_{k=1}^{d}a_{lk}V_{ki_{k}j_{k}}(y_{k})) \\
&=&(\prod
\limits_{l=1}^{d}(\sum \limits_{k=1}^{d}a_{lk}U_{ki_{k}j_{k}}(x_{k})),\prod
\limits_{l=1}^{d}(\sum \limits_{k=1}^{d}a_{lk}V_{ki_{k}j_{k}}(y_{k}))).
\end{eqnarray*}%
Therefore, 
\begin{equation*}
H(t,x,y)=(\prod_{l=1}^{d}(\sum_{k=1}^{d}a_{lk}U_{ki_{k}j_{k}}(x_{k})),%
\prod_{l=1}^{d}((1-t)\sum_{k=1}^{d}a_{lk}V_{ki_{k}j_{k}}(y_{k}))).
\end{equation*}

To prove (\ref{eq:h6}), consider $(x,y)\in M_{c}^{-}$. Then there exists $%
1\leq \beta \leq d$ such that $|x_{\beta }|=1$. Since $\tau $ is a
permutation one can find $\gamma ,1\leq \gamma \leq d$ such that $\tau
(\gamma )=\beta $. By (\ref{eq: thm1}), we have that 
\begin{equation*}
||a_{\gamma \beta }U_{\beta i_{\beta }j_{\beta }}||_{\min }-\sum_{k=1,k\neq
\beta }^{d}||a_{\gamma k}U_{ki_{k}j_{k}}||_{\max }>1\text{.}
\end{equation*}%
Hence, 
\begin{eqnarray}
|\sum_{k=1}^{d}a_{\gamma k}U_{ki_{k}j_{k}}(x_{k})| &\geq &|a_{\gamma \beta
}U_{\beta i_{\beta }j_{\beta }}(x_{\beta })|-\sum_{k=1,k\neq \beta
}^{d}|a_{\gamma k}U_{ki_{k}j_{k}}(x_{k})|  \notag \\
&\geq &||a_{\gamma \beta }U_{\beta i_{\beta }j_{\beta }}||_{\min
}-\sum_{k=1,k\neq \beta }^{d}||a_{\gamma k}U_{ki_{k}j_{k}}||_{\max }  \notag
\\
&>&1.  \label{eq:3}
\end{eqnarray}%
It implies that $H(t,x,y)\notin N_{c}$ and thus (\ref{eq:h6}) holds.

For checking (\ref{eq:h7}), consider $(x,y)\in M_{c}$. Then we get that 
\begin{eqnarray*}
|(1-t)\sum_{k=1}^{d}a_{lk}V_{ki_{k}j_{k}}(y_{k})| &\leq
&|\sum_{k=1}^{d}a_{lk}V_{ki_{k}j_{k}}(y_{k})| \\
&\leq &\sum_{k=1}^{d}||a_{lk}V_{ki_{k}j_{k}}||_{\max } \\
&<&1,
\end{eqnarray*}%
where the last inequality follows from (\ref{eq: thm1}). Therefore $%
H(t,x,y)\notin N_{c}^{+}$ and hence (\ref{eq:h7}) is true.

Next, we check the item 2 in Definition \ref{def: covering relation}.
Consider a map $\varphi :\mathbb{R}^{ud}\rightarrow \mathbb{R}^{ud}$, where 
\begin{equation*}
\varphi (x)=\prod_{l=1}^{d}(\sum_{k=1}^{d}a_{lk}U_{ki_{k}j_{k}}(x_{k})).
\end{equation*}%
Then $H(1,x,y)=(\prod_{l=1}^{d}(%
\sum_{k=1}^{d}a_{lk}U_{ki_{k}j_{k}}(x_{k})),0)=(\varphi (x),0)$. By (\ref%
{eq:3}), we have $\varphi (\partial B^{ud})\subset \mathbb{R}^{ud}\backslash 
\overline{B^{ud}}$.

Finally, we show that the local Brouwer degree\ $\deg (\varphi ,B^{ud},0)$
is nonzero. Observe that we can rewrite 
\begin{equation*}
\varphi (x)=([a_{lk}]\bigotimes I_{u})\circ \prod_{k=1}^{d}U_{ki_{k}j_{k}},
\end{equation*}%
where $I_{u}$\ is the $u\times u$ identity matrix. Since the matrix $%
[a_{lk}] $ is invertible, $[a_{lk}]\bigotimes I_{u}$ is also invertible.
Since $||U_{ki_{k}j_{k}}||_{\min }>1$ and $\deg
(U_{ki_{k}j_{k}},B^{u},0)\neq 0$ for all $1\leq k\leq d$, we have $0\in
U_{ki_{k}j_{k}}(B^{u})$, and hence $0=([a_{lk}]\bigotimes I_{u})^{-1}(0)$
lies in a bounded connected component of $\mathbb{R}^{ud}\backslash
(\prod_{k=1}^{d}U_{ki_{k}j_{k}})(\partial B^{ud})$. By Proposition \ref%
{prop: degree}, we obtain that 
\begin{eqnarray*}
\deg (\varphi ,B^{ud},0) &=&\mathtt{sgn}(\det ([a_{lk}]\bigotimes
I_{u}))\prod_{k=1}^{d}\deg (U_{ki_{k}j_{k}},B^{u},0) \\
&=&\mathtt{sgn}(\det ([a_{lk}]\bigotimes I_{u}))\prod_{k=1}^{d}\deg
(U_{ki_{k}j_{k}},B^{u},0) \\
&\neq &0.
\end{eqnarray*}

We have proved that the needed covering relation holds. If $\tilde{T}_{k}$
and $\tilde{A}$ are both $C^{0}$ close enough to $T_{k}$ and $A$
respectively. Then by Proposition \ref{prop: persistence}, the following
covering relation holds, for all nonzero entries $%
\prod_{k=1}^{d}w_{ki_{k}j_{k}}$ of $\bigotimes_{k=1}^{d}W_{k}$, 
\begin{equation*}
\prod_{k=1}^{d}M_{ki_{k}}\overset{\tilde{A}\circ \tilde{T}}{\Longrightarrow }%
\prod_{k=1}^{d}M_{kj_{k}}.
\end{equation*}%
Therefore, $(\tilde{G},\{ \tilde{T}_{k}\},\tilde{A})$ has covering relations
determined by $\bigotimes_{k=1}^{d}W_{k}$. Since each $W_{k}$ is a
permutation, there exists a closed loop of covering relations for $\tilde{A}%
\circ \tilde{T}$ with loop length $\mathtt{lc}$\texttt{$k$}$(\dim
W_{1},\ldots ,\dim W_{d})$. By Proposition \ref{prop: closed loop}, $\tilde{A%
}\circ \tilde{T}$ has a periodic point of period $\mathtt{lcm}(\dim
W_{1},\ldots ,\dim W_{d})$.
\end{proof}

Next, we prove the second main result.

\begin{proof}[Proof of Theorem \protect \ref{thm: main}]
Let $\prod_{k=1}^{d}w_{ki_{k}j_{k}}$ be a nonzero entry of $%
\bigotimes_{k=1}^{d}W_{k}$. We shall prove that the following covering
relation holds 
\begin{equation*}
\prod_{k=1}^{d}M_{ki_{k}}\overset{A\circ T}{\Longrightarrow }%
\prod_{k=1}^{d}M_{kj_{k}}.
\end{equation*}%
We shall keep the use of the following notations:$\ x_{k}\in \mathbb{R}^{u}$
and $y_{k}\in \mathbb{R}^{s}$ for $1\leq k\leq d$, $x=\prod_{k=1}^{d}x_{k}%
\in (\mathbb{R}^{u})^{d}=\mathbb{R}^{ud}$, $y=\prod_{k=1}^{d}y_{k}\in (%
\mathbb{R}^{s})^{d}=\mathbb{R}^{sd}$, $\prod_{k=1}^{d}(x_{k},y_{k})\in (%
\mathbb{R}^{u+s})^{d}=\mathbb{R}^{(u+s)d}$, and $(x,y)\in \mathbb{R}%
^{ud}\times \mathbb{R}^{sd}=\mathbb{R}^{(u+s)d}$.

First, we check conditions on h-sets. For convenience, we denote $%
M=\prod_{k=1}^{d}M_{ki_{k}}$ and $M^{\prime }=\prod_{k=1}^{d}M_{kj_{k}}$.\
Then $M$ and\ $M^{\prime }$ are h-sets, together with constants $%
u(M)=u(M^{\prime })=ud$ and $s(M)=s(M^{\prime })=sd,$ and homeomorphisms $%
c_{M},c_{M^{\prime }}:\mathbb{R}^{(u+s)d}\rightarrow \mathbb{R}^{ud}\times 
\mathbb{R}^{sd}$, defined as follows, for all $x_{k}\in \mathbb{R}^{u}$ and $%
y_{k}\in \mathbb{R}^{s}$, $1\leq k\leq d,$ 
\begin{equation*}
c_{M}(\prod_{k=1}^{d}(x_{k},y_{k}))=\varsigma \circ \prod_{k=1}^{d}\tilde{c}%
_{M_{ki_{k}}}(x_{k},y_{k}),\text{ and }c_{M^{\prime
}}(\prod_{k=1}^{d}(x_{k},y_{k}))=\varsigma \circ \prod_{k=1}^{d}\tilde{c}%
_{M_{kj_{k}}}(x_{k},y_{k}),
\end{equation*}%
where $\varsigma :\mathbb{R}^{(u+s)d}\rightarrow \mathbb{R}^{(u+s)d}$ is
defined by $\varsigma (\prod_{k=1}^{d}(x_{k},y_{k}))=(x,y)$.

Second, we construct a homotopy such that equations (\ref{eq:h5})-(\ref%
{eq:h7}) holds. Define a homotopy $H:[0,1]\times \overline{B^{ud}}\times 
\overline{B^{sd}}\rightarrow \mathbb{R}^{(u+s)d}$ by, if $0\leq t\leq 1/2$,
then

\begin{equation*}
H(t,x,y)=(1-2t)c_{M^{\prime }}\circ A\circ T\circ c_{M}^{-1}(x,y)+2t\pi
_{1}\circ c_{M^{\prime }}\circ A\circ T\circ c_{M}^{-1}(x,y),
\end{equation*}%
and if $1/2<t\leq 1$, then%
\begin{equation*}
H(t,x,y)=(2-2t)\pi _{1}\circ c_{M^{\prime }}\circ A\circ T\circ
c_{M}^{-1}(x,y)+(2t-1)(\prod_{k=1}^{d}(a_{k\tau (k)}U_{\tau (k)i_{\tau
(k)}}(x_{\tau (k)})-p_{kj_{k}}^{u}),0).
\end{equation*}%
where $\pi _{1}:\mathbb{R}^{(u+s)d}\rightarrow \mathbb{R}^{(u+s)d}$ is
defined by $\pi _{1}(x,y)=(x,0)$ for all $x\in \mathbb{R}^{ud}$ and $y\in 
\mathbb{R}^{sd}$. Clearly, (\ref{eq:h5}) holds.

Before checking (\ref{eq:h6}) and (\ref{eq:h7}), we derive a new form for
the homotopy. Define $\breve{c}_{N}=(\prod_{k=1}^{d}\hat{c}_{N_{k}})$. Let $%
(x,y)\in \overline{B^{ud}}\times \overline{B^{sd}}=M_{c}$. Then by the
definitions of $\tilde{c}_{N_{kj_{k}}}$, we get that 
\begin{equation*}
\breve{c}_{N}\circ T\circ \prod_{k=1}^{d}\tilde{c}%
_{N_{kj_{k}}}^{-1}(x_{k},y_{k})=%
\prod_{k=1}^{d}(U_{ki_{k}}(x_{k}),V_{ki_{k}}(y_{k}))\in c_{N}(T(M)).
\end{equation*}%
Thus, 
\begin{eqnarray*}
c_{M^{\prime }}\circ A\circ T\circ c_{M}^{-1}(x,y) &=&\varsigma \circ
(\prod_{l=1}^{d}\tilde{c}_{N_{lj_{l}}})\circ A\circ T\circ (\prod_{k=1}^{d}%
\tilde{c}_{M_{kj_{k}}}^{-1})\circ \varsigma ^{-1}(x,y) \\
&=&\varsigma \circ (\prod_{l=1}^{d}\tilde{c}_{N_{lj_{l}}})\circ A\circ 
\breve{c}_{N}^{-1}\circ \breve{c}_{N}\circ T\circ \prod_{k=1}^{d}\tilde{c}%
_{M_{kj_{k}}}^{-1}(x_{k},y_{k}) \\
&=&\varsigma \circ (\prod_{l=1}^{d}g_{lj_{l}}\circ \hat{c}_{N_{l}})\circ
A\circ \breve{c}_{N}^{-1}\circ
\prod_{k=1}^{d}(U_{ki_{k}}(x_{k}),V_{ki_{k}}(y_{k})) \\
&=&\varsigma \circ (\prod_{l=1}^{d}g_{lj_{l}})\circ \breve{c}_{N}\circ
A\circ \breve{c}_{N}^{-1}\circ
\prod_{k=1}^{d}(U_{ki_{k}}(x_{k}),V_{ki_{k}}(y_{k})).
\end{eqnarray*}%
Moreover, by the definitions of $A_{c}$ and $g_{lj_{l}}$, we obtain that 
\begin{eqnarray*}
c_{M^{\prime }}\circ A\circ T\circ c_{M}^{-1}(x,y) &=&\varsigma \circ
(\prod_{l=1}^{d}g_{lj_{l}})\circ A_{c}\circ
\prod_{k=1}^{d}(U_{ki_{k}}(x_{k}),V_{ki_{k}}(y_{k})) \\
&=&\varsigma \circ (\prod_{l=1}^{d}g_{lj_{l}})\circ ([a_{lk}]\bigotimes
I)\circ \prod_{k=1}^{d}(U_{ki_{k}}(x_{k}),V_{ki_{k}}(y_{k})) \\
&=&\varsigma \circ
\prod_{l=1}^{d}(\sum_{k=1}^{d}a_{lk}U_{ki_{k}}(x_{k})-p_{lj_{l}}^{u},\frac{%
\sum_{k=1}^{d}a_{lk}V_{ki_{k}}(y_{k})-p_{lj_{l}}^{s}}{r_{lj_{l}}}) \\
&=&(\prod_{l=1}^{d}(\sum_{k=1}^{d}a_{lk}U_{ki_{k}}(x_{k})-p_{lj_{l}}^{u}),%
\prod_{l=1}^{d}(\frac{\sum_{k=1}^{d}a_{lk}V_{ki_{k}}(y_{k})-p_{lj_{l}}^{s}}{%
r_{lj_{l}}})).
\end{eqnarray*}%
Therefore, for $0\leq t\leq 1/2,$ 
\begin{equation*}
H(t,x,y)=(\prod_{l=1}^{d}(%
\sum_{k=1}^{d}a_{lk}U_{ki_{k}}(x_{k})-p_{lj_{l}}^{u}),(1-2t)\prod_{l=1}^{d}(%
\frac{\sum_{k=1}^{d}a_{lk}V_{ki_{k}}(y_{k})-p_{lj}^{s}}{r_{lj_{l}}})),
\end{equation*}%
and, for $1/2<t\leq 1,$%
\begin{equation*}
H(t,x,y)=(\prod_{l=1}^{d}(a_{l\tau (l)}U_{\tau (l)i_{\tau (l)}}(x_{\tau
(l)})-p_{lj_{l}}^{u}+(2-2t)\sum_{k=1,k\neq \tau
(l)}^{d}a_{lk}U_{ki_{k}}(x_{k})),0).
\end{equation*}

For checking (\ref{eq:h6}), consider $(x,y)\in M_{c}^{-}$. Then there exists 
$1\leq \beta \leq d$ such that $|x_{\beta }|=1$. Since $\tau $ is a
permutation, there exists $1\leq \gamma \leq d$ such that $\tau (\gamma
)=\beta $. By (\ref{eq: thm2}), we have that 
\begin{equation*}
||a_{\gamma \beta }U_{\beta i_{\beta }}-p_{\gamma j_{\gamma }}^{u}||_{\min
}-\sum_{k=1,k\neq \beta }^{d}||a_{\gamma k}U_{ki_{k}}||_{\max }>1\text{.}
\end{equation*}%
It implies that 
\begin{eqnarray}
|\sum_{k=1}^{d}a_{\gamma k}U_{ki_{k}}(x_{k})-p_{\gamma j_{\gamma }}^{u}|
&\geq &|a_{\gamma \beta }U_{\beta i_{\beta }}(x_{\beta })-p_{\gamma
j_{\gamma }}^{u}|-\sum_{k=1,k\neq \beta }^{d}|a_{\gamma k}U_{ki_{k}}(x_{k})|
\notag \\
&\geq &||a_{\gamma \beta }U_{\beta i_{\beta }}-p_{\gamma j_{\gamma
}}^{u}||_{\min }-\sum_{k=1,k\neq \beta }^{d}||a_{\gamma k}U_{ki_{k}}||_{\max
}  \notag \\
&>&1.  \label{eq:5}
\end{eqnarray}%
and, for $1/2<t\leq 1,$ 
\begin{eqnarray*}
&&|a_{\gamma \beta }U_{\beta i_{\beta }}(x_{\beta })-p_{\gamma j_{\gamma
}}^{s}+(2-2t)\sum_{k=1,k\neq \beta }^{d}a_{\gamma k}U_{ki_{k}}(x_{k})| \\
&\geq &|a_{\gamma \beta }U_{\beta i_{\beta }}(x_{\beta })-p_{\gamma
j_{\gamma }}^{u}|-(2-2t)\sum_{k=1,k\neq \beta }^{d}|a_{\gamma
k}U_{ki_{k}}(x_{k})| \\
&\geq &||a_{\gamma \beta }U_{\beta i_{\beta }}-p_{\gamma j_{\gamma
}}^{u}||_{\min }-\sum_{k=1,k\neq \beta }^{d}||a_{\gamma k}U_{ki_{k}}||_{\max
} \\
&>&1.
\end{eqnarray*}%
Thus $H(t,x,y)\notin M_{c}^{\prime }$ and hence (\ref{eq:h6}) holds.

For checking (\ref{eq:h7}), consider $(x,y)\in M_{c}$. Then, for $0\leq
t\leq 1/2,$ we have that 
\begin{eqnarray*}
|(1-2t)\frac{\sum_{k=1}^{d}a_{lk}V_{ki_{k}}(y_{k})-p_{lj}^{s}}{r_{lj_{l}}}|
&\leq &\frac{1}{r_{lj_{l}}}%
|\sum_{k=1}^{d}a_{lk}V_{ki_{k}}(y_{k})-p_{lj_{l}}^{s}| \\
&\leq &\frac{1}{r_{lj_{l}}}%
\sum_{k=1}^{d}||a_{lk}V_{ki_{k}}-p_{lj_{l}}^{s}||_{\max } \\
&<&1.
\end{eqnarray*}%
where the last inequality follows from (\ref{eq: thm2}). Thus, $%
H(t,x,y)\notin N_{c}^{+}$ and hence (\ref{eq:h7}) holds.

Next, we check item 2 in Definition \ref{def: covering relation}. Define a
map $\varphi :\mathbb{R}^{ud}\rightarrow \mathbb{R}^{ud}$ by, 
\begin{equation*}
\varphi (x)=\prod_{l=1}^{d}(a_{l\tau (l)}U_{\tau (l)i_{\tau (l)}}(x_{\tau
(l)})-p_{lj_{l}}^{u}).
\end{equation*}%
Then $H(1,x,y)=(\prod_{l=1}^{d}(a_{l\tau (l)}U_{\tau (l)i_{\tau
(l)}}(x_{\tau (l)})-p_{lj_{l}}^{u}),0)=(\varphi (x),0)$. By equation (\ref%
{eq:5}), we have $\varphi (\partial B^{ud})\subset \mathbb{R}^{ud}\backslash 
\overline{B^{ud}}$.

Finally, we prove that the local Brouwer degree\ $\deg (\varphi ,B^{ud},0)$
is nonzero. Define a function $g(x)=\prod_{l=1}^{d}(x_{l}-p_{lj_{l}}^{u})$
for $x\in \mathbb{R}^{ud}$, and a $d\times d$ matrix $[b_{lk}]$ such that $%
b_{l\tau (l)}=a_{l\tau (l)}$ and $b_{lk}=0$ for all $k\neq \tau \left(
l\right) $. Then 
\begin{equation*}
\varphi (x)=g\circ ([b_{lk}]\bigotimes I_{u})\circ
\prod_{k=1}^{d}U_{ki_{k}}(x_{k}),
\end{equation*}%
where $I_{u}$\ is the $u\times u$ identity matrix. In order to apply
Proposition \ref{prop: degree} for $\deg (\varphi ,B^{ud},0)$, we need to
check conditions on the affine maps $g$ and $[b_{lk}]\bigotimes I_{u}.$ By
the definition of $g$, we get that $g^{-1}(0)$ consists of a single point $%
p\equiv \prod_{l=1}^{d}p_{lj_{l}}^{u}$. By the definition of $[b_{lk}]$ and $%
b_{l\tau (l)}=a_{l\tau (l)}$, the hypothesis $p_{lj_{l}}^{u}\in a_{l\tau
(l)}U_{\tau (l)i_{\tau (l)}}(B^{u})$, together with (\ref{eq: thm2}),
implies that the matrix $[b_{lk}]$ is invertible; otherwise, $||a_{l\tau
(l)}U_{\tau (l)i_{\tau (l)}}-p_{lj_{l}}^{u}||_{\min }=0$ leads a
contradiction. Hence, $([b_{lk}]\bigotimes I_{u})^{-1}(p)$ lies in a bounded
connected component of $\mathbb{R}^{ud}\backslash
(\prod_{k=1}^{d}U_{ki_{k}})(\partial B^{ud})$. Since $\deg
(U_{li_{l}},B^{u},p_{lj_{l}}^{u})\neq 0$, we have $p_{lj_{l}}^{u}\in
U_{li_{l}}(B^{u})$ and hence $p\in \prod_{l=1}^{d}U_{li_{l}}(B^{u})$. By
applying Proposition \ref{prop: degree}, since $\deg
(U_{li_{l}},B^{u},p_{lj_{l}}^{u})\neq 0$ for all $1\leq l\leq d$, we obtain
that 
\begin{eqnarray*}
\deg (\varphi ,B^{ud},0) &=&\mathtt{sgn}(\det (Dg_{p}))\mathtt{sgn}(\det
([b_{lk}]\bigotimes I_{u}))\deg (\prod_{l=1}^{d}U_{li_{l}},B^{ud},p) \\
&=&\mathtt{sgn}(\det (Dg_{p}))\mathtt{sgn}(\det ([b_{lk}]\bigotimes
I_{u}))\prod_{l=1}^{d}\deg (U_{li_{l}},B^{u},p_{lj_{l}}^{u}) \\
&\neq &0.
\end{eqnarray*}

This concludes the proof of the needed covering relation. If $\tilde{T}_{k}$
and $\tilde{A}$ are both $C^{0}$ close enough to $T_{k}$ and $A$
respectively, then by Proposition \ref{prop: persistence}, the following
covering relation holds, for all nonzero entries $%
\prod_{l=1}^{d}w_{li_{l}j_{l}}$ of $W\equiv \bigotimes_{l=1}^{d}W_{l}$, 
\begin{equation*}
\prod_{l=1}^{d}M_{li_{l}}\overset{\tilde{A}\circ \tilde{T}}{\Longrightarrow }%
\prod_{l=1}^{d}M_{lj_{l}}.
\end{equation*}%
Therefore, $(\tilde{G},\{ \tilde{T}_{k}\},\tilde{A})$ has covering relations
determined by $W$. By Proposition \ref{prop: semi-conjugacy}, there exists a
compact subset $\tilde{\Lambda}$\ of $\mathbb{R}^{(u+s)d}$\ such that $%
\tilde{\Lambda}$\ is a maximal positively invariant for $\tilde{A}\circ 
\tilde{T}$\ in the union of the h-sets (with respect to $W$) and $\tilde{A}%
\circ \tilde{T}|\tilde{\Lambda}$\ is topologically semi-conjugate to $\sigma
_{W}^{+}$. Therefore, 
\begin{equation*}
h_{\mathtt{top}}(\tilde{A}\circ \tilde{T})\geq h_{\text{$\mathtt{top}$}%
}(\sigma _{W}^{+})=\mathtt{\log }(\rho (W)=\mathtt{\log }(\prod_{l=1}^{d}%
\rho (W_{l})).
\end{equation*}
\end{proof}

\noindent \textbf{\textbf{\Large Appendix}}

\medskip

\noindent First, we briefly recall some definitions from \cite{ZG04}
concerning covering relations.

\begin{definition}
\cite[Definition 6]{ZG04} \label{def: h-sets} An \emph{h-set} in $\mathbb{R}%
^{m}$ is a quadruple consisting of the following data:

\begin{itemize}
\item a nonempty compact subset $M$ of $\mathbb{R}^{m},$

\item a pair of numbers $u(M),s(M)\in \{0,1,...,m\}$ with $u(M)+s(M)=m,$

\item a homeomorphism $c_{M}:\mathbb{R}^{m}\rightarrow \mathbb{R}^{m}=%
\mathbb{R}^{u(M)}\times \mathbb{R}^{s(M)}$ with $c_{M}(M)=\overline{B^{u(M)}}%
\times \overline{B^{s(M)}},$ where $S\times T$ is the Cartesian product of
sets $S$ and $T$.
\end{itemize}

\noindent For simplicity, we will denote such an h-set by $M$ and call $%
c_{M} $ the \emph{coordinate chart} of $M$; furthermore, we use the
following notations: 
\begin{equation*}
M_{c}=\overline{B^{u(M)}}\times \overline{B^{s(M)}},\text{ }%
M_{c}^{-}=\partial B^{u(M)}\times \overline{B^{s(M)}},\text{ }M_{c}^{+}=%
\overline{B^{u(M)}}\times \partial B^{s(M)},
\end{equation*}%
\begin{equation*}
M^{-}=c_{M}^{-1}(M_{c}^{-}),\text{ and }M^{+}=c_{M}^{-1}(M_{c}^{+}).
\end{equation*}
\end{definition}

A covering relation between two h-sets is defined as follow.\textbf{\ }

\begin{definition}
\cite[Definition 7]{ZG04}\label{def: covering relation} Let $M,$ $N$ be
h-sets in $\mathbb{R}^{m}$ with $u(M)=u(N)=u$ and $s(M)=s(N)=s,$ $%
f:M\rightarrow \mathbb{R}^{m}$ be a continuous map, and $f_{c}=c_{N}\circ
f\circ c_{M}^{-1}:M_{c}\rightarrow \mathbb{R}^{u}\times \mathbb{R}^{s}$. We
say $M$\emph{\ }$f$\emph{-covers\ }$N$, and write 
\begin{equation*}
M\overset{f}{\Longrightarrow }N,
\end{equation*}%
\textbf{\ }if the following conditions are satisfied:

\begin{enumerate}
\item there exists a homotopy $h:[0,1]\times M_{c}\rightarrow \mathbb{R}%
^{u}\times \mathbb{R}^{s}$ such that 
\begin{eqnarray}
h(0,x) &=&f_{c}(x)\text{ for }x\in M_{c},  \label{eq:h5} \\
h([0,1],M_{c}^{-})\cap N_{c} &=&\emptyset ,  \label{eq:h6} \\
h([0,1],M_{c})\cap N_{c}^{+} &=&\emptyset ;  \label{eq:h7}
\end{eqnarray}

\item there exists a map $\varphi :\mathbb{R}^{u}\rightarrow \mathbb{R}^{u}$
such that 
\begin{align*}
h(1,p,q)& =(\varphi (p),0)\text{ for any }p\in \overline{B^{u}}\text{ and }%
q\in \overline{B^{s}}\text{,} \\
\varphi (\partial B^{u})& \subset \mathbb{R}^{u}\backslash \overline{B^{u}};%
\text{ and }
\end{align*}

\item there exists a nonzero integer $w$ such that the local Brouwer degree $%
\deg (\varphi ,B^{u},0)$ of $\varphi $ at $0$ in $B^{u}$ is $w$; refer to 
\cite[Appendix]{ZG04} for its properties.
\end{enumerate}
\end{definition}

A \emph{transition matrix} is a square matrix which satisfies the following
conditions:\ 

\begin{description}
\item[(i)] all entries are either zero or one,

\item[(ii)] all row sums and column sums are greater than or equal to one.
\end{description}

\noindent For a transition matrix $W$, let $\rho (W)$ denote the spectral
radius of $W$. Then $\rho (W)\geq 1$ and moreover, if $W$ is irreducible and
not a permutation, then $\rho (W)>1$. Let $\Sigma _{W}^{+}$ (resp. $\Sigma
_{W}$) be the space of all allowable one-sided (resp. two sided)\ sequences
generated by the transition matrix $W$ with a usual metric, and let $\sigma
_{W}^{+}:\Sigma _{W}^{+}\rightarrow \Sigma _{W}^{+}$ (resp. $\sigma
_{W}:\Sigma _{W}\rightarrow \Sigma _{W}$) be the one-sided (resp. two sided)
subshift of finite type for $W$. Then $h_{\mathtt{top}}(\sigma _{W}^{+})=h_{%
\mathtt{top}}(\sigma _{W})=\log (\rho (W))$ (Refer to \cite{R99} for more
background).

\begin{definition}
\label{def:subshift and covering relation} Let $W=[w_{ij}]_{1\leq i,j\leq
\gamma }$ be a transition matrix and $f$ be a continuous map on $\mathbb{R}%
^{m}$. We say that $f$ has \emph{covering relations determined by} $W$ if
the following conditions are satisfied:

\begin{enumerate}
\item there are $\gamma $ pairwisely disjoint h-sets $\{M_{i}\}_{i=1}^{%
\gamma }$ in $\mathbb{R}^{m}$;

\item if $w_{ij}=1$ then the covering relation $M_{i}\overset{f}{%
\Longrightarrow }M_{j}$ holds;
\end{enumerate}
\end{definition}

\end{document}